\documentclass[11pt,reqno]{amsart}
\usepackage{amsmath, fullpage}
\usepackage{paralist}
 \usepackage[colorlinks=true]{hyperref}
\hypersetup{urlcolor=blue, citecolor=blue, linkcolor=blue}
 
\newcommand{\doi}[1]{\href{http://dx.doi.org/#1}{\tiny doi:#1}}

\newtheorem{theorem}{Theorem}
\newtheorem{corollary}[theorem]{Corollary}

\newtheorem{lemma}[theorem]{Lemma}

\theoremstyle{definition}

\newtheorem{remark}{Remark}

\def\mR{{\mathbb R}}
\def\mT{{\mathbb T}}
\def\mS{{\mathbb S}^2}
\def\mM{{\mathbb M}}
\def\ssperp{{{\perp}}}
\def\Deltainv {\Delta^{\!\scriptstyle{-1}}}
\def\vu{{\mathbf{u}}}
\def\vv{{\mathbf{v}}}

\def\irr{\textnormal{irr}}
\def\inc{\textnormal{inc}}
\def\vuirr{\vu_\irr}\def\vuinc{\vu_\inc}

\def\ep{\varepsilon}
\def\epF{\delta}
\def\epR{\ep}
\def\pa{\partial}

\def\bpm{\begin{pmatrix}}
\def\epm{\end{pmatrix}}

\def\be{\begin{equation}}
\def\ee{\end{equation}}
\def\dv{\textnormal{\,div\,}}
\def\curl{\textnormal{\,curl\,}}
\def\grad{\nabla}
\def\rgrad{\nabla^{\ssperp}}
\def\pt{\pa_t}

\def\cA{{\mathcal A}}
\def\cL{{\mathcal L}}
\def\cN{{\mathcal N}}
\def\cH{{\mathcal H}}
\def\cG{{\mathcal G}}

\def\cAs{{\cA^s}}

\def\fCJ{\mathfrak{C}_J}

\def\fco{{\mathsf F}}
\def\nll{\textnormal{ker}}
\def\nllL{\nll\{\!\cL\!\}}
\def\pcL{{\displaystyle\operatorname*{\displaystyle\Pi}_{\scriptscriptstyle\nllL}}}
\def\vutil{\widetilde{\vu}}

\def\cn{\!\cdot\!\!\nabla}

\def\ba{\begin{aligned}}
\def\ea{\end{aligned}}

\def\ccdot{{\hspace{-0.15mm}\cdot }}

\def\sg{\mathfrak g}

\def\pa{\partial}
\def\pt{\pa_t}

 \def\cN{\mathcal N}

\def\qand{\quad\text{and}\quad}
\def\qfor{\quad\text{for}\quad}

\def\Vcol{\bpm\vu\\h\epm}
\def\ti{\tilde}

\title[SWE rotating surface]{Shallow water equations on a fast rotating surface}
\author{\bf Bin Cheng}
\address{Department of Mathematics, University of Surrey, Guildford, GU2 7XH, United Kingdom}
\email{b.cheng@surrey.ac.uk}

\author{\bf Steve Schochet}
\address{School of Mathematical Sciences, Tel-Aviv University, Tel Aviv 69978, Israel}
  \email{schochet@tauex.tau.ac.il}
  
\keywords{
Rotating shallow water equations,
PDEs on manifolds,
PDEs on surfaces,
 singular limits,
 uniform estimates,
 commutator,
 variable Coriolis parameter,
zonal flows,
  time-averages
}  

\date{16 July 2019}
 
\begin{document}
\maketitle
 
 \begin{abstract}
 We prove that for rotating shallow water equations on a surface of revolution with variable Coriolis parameter and vanishing Rossby and Froude numbers, the classical solution satisfies uniform estimates on a fixed time interval with no dependence on the small parameters. Upon a transformation using the solution operator associated with the large operator, the solution converges strongly to a limit for which the governing equation is given. We also characterize the kernel of the large operator and define a projection onto that kernel. With these tools, we are able to show that the time-averages of the solution are close to longitude-independent zonal flows and height field.
 \end{abstract}
 
 \bigskip
 
\section{\bf Introduction}
We investigate \emph{compressible} fluid dynamics on a  two-dimensional, smooth manifold $\mM$ which is a surface of revolution about the $z$ axis, generated by a curve connecting the north/south poles at $(0,0,\pm1)\in\mR^3$. The fluid is subject to the Coriolis force and we use a notation $J$ to denote the clockwise rotation operator on any vector field $\vu\in T\mM$ (tangent bundle of $\mM$), namely $J\vu$ is defined as the cross product of $\vu$ and the outward normal to the surface. 
A popular model from geophysics is the rotating shallow water (RSW) equations, 
\begin{subequations}\label{RSW}
\begin{align}
\label{RSW:u}\pt \vu+\nabla_\vu\vu=&-{1\over\epF}\grad h-{\fco\over\epR}J\vu\\
\label{RSW:h}\pt h+\nabla_\vu h+h\dv \vu=&-{1\over\epF}\dv\vu
\end{align}
\end{subequations}
where $\epF$ denotes the Froude number and $\epR$ the {\it planetary}  Rossby number. The variable $h$ denotes perturbation of height against the background rescaled to 1, so that the total height is $1+\epF h$. The notation $\nabla_\vu h$ is understood simply as $\vu\cn h$ because $h$ is a scalar defined everywhere on the surface. The notation $\nabla_\vu\vu$ denotes the covariant derivative of $\vu$ along the vector field $\vu$ and can be understood\footnote{The intrinsic definition of covariant derivatives depend on the provision of connections or the full curvature tensor.  It can also be described extrinsically, if the surface is embedded in a usual Euclidean space and we allow that Euclidean space to provide all information needed on the metric and curvature tensor. We use the latter approach here.} as the result of projecting the  Cartesian form $(\nabla_\vu u_1,\nabla_\vu u_2,\nabla_\vu u_3)$ onto $T\mM$. Finally, the scalar factor $\fco$ represents variation in the Coriolis parameter. For physicality, $\fco$ takes value 1 at the north pole and $-1$ at the south pole.

This system is used as a standard model for testing numerical code on spherical domains \cite{RSW:testcases}.


We define the following notation for the large operator  on the right side of \eqref{RSW}, 
\be\label{def:cM}
\ba 
\cL&:=\frac{1}{\epF}\cL_\pa+{1\over\epR}\cL_0\\
\text{where } \qquad\cL_\pa\Vcol&:=\bpm \grad h\\\dv \vu\epm \qand \cL_0\Vcol:=\bpm \fco J\vu\\0\epm.
\ea
\ee
 
Without loss of generality, we also impose zero global mean condition on $h$
\be\label{h:global:mean}\nonumber\int_{\mS}h(t,x)\,dx=0\quad\mbox{ for all times }t\ee
since for smooth solutions, the global mean of $h$ is time-invariant due to the conservation law \eqref{RSW:h}.

Surface $\mM$ is parametrised by longitude-colatitude pair $(p_1,p_2)$ as follows.
The axis of revolution is the $z$ axis. The generating curve 
is given in  Cartesian coordinates as $(x,y,z)=(\sin p_2,0,\cos p_2)R(p_2)$ for $p_2\in[0,\pi]$.
Then, the  surface of revolution $\mM$ is parametrized in Cartesian coordinates as 
\be\label{surface:par}
\bpm
x\\y\\z
\epm=
R(p_2)\bpm
\cos p_1\sin p_2\\\sin p_1\sin p_2\\\cos p_2
\epm \qfor (p_1,p_2)\in\mT_{2\pi}\times [0,\pi],
\ee
where $\mT_{2\pi}$ denotes the one-dimensional torus of length $2\pi$ and function $R\in C^\infty([0,\pi]) $ satisfies
\[
\inf_{[0,\pi]}R(\cdot)>0,\quad R(0)=R(\pi)=1.
\]
 The exact spherical domain is represented by $R(p_2)\equiv1$.  There are further mild conditions on $R$ to ensure $\mM$ is a reasonable manifold to work with, and their details are given in \eqref{cond:chart}, \eqref{cond:Rm} and discussed therein. Also, the detailed geometry of $\mM$ is discussed in Section \ref{sec:geo} in elementary terms. For now, we only need to define $g_1,g_2$ to be the diagonal entries of the metric tensor  \eqref{metric:gij} namely
 \[
 g_1(p_2):=\sin^2(p_2)R^2(p_2)\qand
g_2(p_2):=R^2(p_2)+\left(R'(p_2)\right)^2.
 \]

Since we consider small values of $\epF,\epR$ and the limiting solutions and equations when they approach zero, it is crucial to prove that the time interval of validity of our results is uniformly bounded from below, independent of the smallness of the parameters. In this article, classical solutions with at least $C^1(\mM)$ regularity are considered and we rely on energy method/estimate to achieve that regularity, and thus we will actually deal with $H^k(\mM)$ norms  for $k>2$. The coefficients in the large operator  $\cL$ varies with colatitude,  which imposes a major challenge in such endeavor. The examples in \S \ref{subs:loss} show that not every operator that is skew-self-adjoint in $L^2$ can guarantee the life span of classical solutions to be uniformly bounded from below.

We resolve this issue in the next theorem, which is the first in this kind of uniform estimates on an entire, non-flat manifold. 
\begin{theorem}\label{thm:main}
Let integer $k>2$.
Consider the rotating shallow water equations \eqref{RSW} on the spatial domain $\mM$ which is a surface of revolution parametrised by \eqref{surface:par} satisfying smoothness conditions \eqref{cond:chart}, \eqref{cond:Rm}. Suppose $\epF\le C\epR$ and the initial data is uniformly bounded in $H^k$. 

If the Coriolis parameter  $\fco\in C^k(\mM)$ satisfies $\pa_{1}\fco=0$ and, for some $\epF,\epR$-independent constants $C_\fco,C_\fco'$,
 \[
 \Big\|{(\pa_2\fco)\sqrt{g_1g_2}-C_\fco g_1g_2\over g_2}\Big\|_{H^k}\le C_\fco'\ep,
 \] then the solution $(\vu,h)$ satisfies an $\epF,\epR$-independent $H^k$ bound over a time interval $[0,T_{m}]$ that is also $\epF,\epR$-independent.
\end{theorem}
In this article, we adopt definition \eqref{def:Hk} for the $H^k$ norm. The subscript $k$ in notation $\|\,\|_k$ is short for $H^k$.

The proof combines Lemma \ref{lem:exact:com} and standard energy estimates. When the domain $\mM$ is a perfect sphere, i.e. $R\equiv1$ in \eqref{surface:par} so that $g_1=\sin^2 p_2$ and $g_2=1$, the standard geophysical model (\cite{White:etal}) ensures that $\fco=\cos p_2$ which validates the assumption on $\fco$ in the theorem with $C_\fco=-1$ and $C_\fco'=0$. 
When the domain is not a perfect sphere, the approximation argument made in  \cite{White:etal} no longer has a clear generalisation. Of course, it is still reasonable to assume that if $\mM$ is an $O(\ep)$ perturbation of a perfect sphere, then $\fco$ is an $O(\ep)$ perturbation of $\cos p_2$, which  validates the assumption in the above theorem.

\begin{corollary}
Under the same assumptions as Theorem \ref{thm:main} with $C_\fco'\equiv0$,  the operator exponential $e^{t\cL}$ is well-defined for any $t\in\mR$ and is bounded mapping from $H^k$ to $H^k$. Moreover, let
\[
{V_{\epR,\epF}}:=e^{t\cL}\Vcol(t,\cdot).
\]
and fix $\dfrac{\epF}{\epR}= \mu>0$. Then, as $\epF\to0$, the transformed solution $V_{\epR,\epF}$ tends strongly in
$C([0,T_m];H^{k'}(\mM))$ space with $k'<k$  to $\bar V$ that  solves the following  equation with the same initial data as $V_{\epR,\epF}$
\begin{align}\label{eq:barV}
&\pt \bar V+\bar B(\bar V)=0, \\
\label{def:barB}&\bar B(\bar V(t,x)):=\lim_{\ell\to\infty}{1\over \ell}\int_0^\ell e^{s(\cL_\pa+\mu\cL_0)}{B}\left[e^{-s(\cL_\pa+\mu\cL_0)}\bar V(t,x)\right]\, ds  
\end{align}
where $B\left[\Vcol\right]:=\begin{pmatrix}\nabla_\vu \vu\\\nabla_\vu h+h\dv \vu\end{pmatrix}$ comes from the original PDE \eqref{RSW}.
\end{corollary}

\begin{proof}
By Lemma \ref{lem:exact:com}, $(\Delta-\cN)$ and $\cL$ commute with  $\cN$ being a first order self-adjoint differential operator.
By classical spectral theory of linear operators on Hilbert space (e.g. \cite{Taylor:II}), the eigenfunctions of $\Delta$ can form
a complete basis of $L^2(\mM)$. Since $\Delta-\cN$ is just a lower order perturbation, it also has eigensfunctions that  form a complete basis of $L^2(\mM)$. The commutability of $(\Delta-\cN)$ and $\cL$ together with the skew-self-adjointness of $\cL$ implies that $\cL$ share the same eigenfunctions with zero or pure imaginary eigenvalues.

Further, because $H^k(\mM)$ and $(-\Delta+\cN)^{-{k\over2}} L^2(\mM)$ are the same space and because spectral theory guarantees the above eigenfunctions are smooth,  they also form a complete basis of $H^k(\mM)$.  Upon the correct normalisation in the respective $H^k(\mM)$ norm, they form an orthonormal basis. 

 This means we can define operator exponential $e^{t\cL}$ as bounded mapping on $L^2(\mM)$ and any $H^k(\mM)$ using the same eigenfunction expansion. 
 
Recall $\cL$  only has zero or pure imaginary eigenvalues. Therefore,  
\be\label{def:W}
W(s,t)=e^{s(\cL_\pa+\mu\cL_0)}\begin{pmatrix}\vu(t,x)\\h(t,x)\end{pmatrix}
\ee 
is an almost periodic function in the $s$ variable  when $W$ is considered as a mapping from the pair $(s,t)$ to  points in $H^k(\mM)$. The classical theory  of almost periodic functions (e.g. \cite{Bes:AP})  guarantees that the integrand on the right hand side of \eqref{def:barB} is almost periodic in the $s$ variable and therefore the limit therein exists. Then Schochet's construction of proof in \cite{Sch:jde} for singular limits of PDEs on $\mT^n$ or $\mR^n$ spatial domain can be adopted here. In particular, 
\begin{enumerate}[(i)]
\item the limit of $V_{\epR,\epF}=\bar V$ exists strongly in $C([0,T_m], H^{k'}(\mM))$ ($k'<k$) by compactness argument, up to choosing a subsequence (but see (\ref{uniquelimit}) below);
\item the time integral of the transformed original PDE
\[
 \int_0^TV_{\epR,\epF}(t)\,dt=\int_0^T  e^{{t\over\epF}\cL_\pa+{t\over\epR}\cL_0}{B}\left[e^{-{t\over\epF}\cL_\pa -{t\over\epR}\cL_0}V_{\epR,\epF}\right]\,dt \qfor T\in(0,T_m]
\] can be approximated as
\[
 \int_0^TV_{\epR,\epF}(t)\,dt=\int_0^T  e^{{t\over\epF}(\cL_\pa+\mu\cL_0)}{B}\left[e^{-{t\over\epF}(\cL_\pa+\mu\cL_0)}
 \bar V(t) \right]\,dt+o(1)T 
\]
and then, by the Krylov-Bogoliubov-Mitropolsky averaging method (\cite{BM:ave}),
satisfies the following limit strongly in $H^{k-1}(\mM)$, up to choosing a subsequence (but see (\ref{uniquelimit}) below),
\[\lim_{\epF\to0}\int_0^TV_{\epR,\epF}(t)\,dt=\int_0^T \left(\lim_{\ell\to\infty}{1\over \ell}\int_0^\ell  e^{s(\cL_\pa+\mu\cL_0)}{B}\left[e^{-s(\cL_\pa+\mu\cL_0)}\bar V(t)\right]\,ds\right)\,dt;
\]
\item \label{uniquelimit} the strong limit of any subsequence of $ V_{\epR,\epF}$ must satisfy \eqref{eq:barV}, \eqref{def:barB} which have a unique solution with given initial data $\bar V_0=(\vu_0,h_0)$ and therefore strong limit $\displaystyle\lim_{\epF\to0}V_{\epR,\epF}$ exists uniquely.
\end{enumerate}
\end{proof}

The $\beta$-effect i.e. variation of Coriolis parameter $\fco$ is most prominent about the equator, and upon linearisation, it is approximately proportional to the signed distance of the point to the equator.
The results in \cite{DM:limit:2006, DM:ave:2007, GS:2006, DSM:simple:2009} adopt such linearisation, set the domain in flat 2D space which {\it extends to infinity} along the north-south direction. The slow subspace of solutions contains zonal flows which coincide with Theorem \ref{thm:TA} from below. 
Also note that in
\cite{DM:ave:2007}, \cite{GS:2006}, solution space is expanded using Hermite functions which are essentially Gaussian functions times polynomials.

A straightforward framework is introduced in \cite{ChMa:TA} for proving that the time-average of the solution stays close to the null space of the large linear operator. 
An application of this  framework can be found in \cite{ChMa:Euler:sphere} for two-dimensional {\it incompressible} Euler equations on a fast rotating sphere -- whereas \eqref{RSW} is a special case of two-dimensional {\it compressible} Euler equations. Another application is in \cite{CM19} in the domain of a thin spherical shell. 

Our next theorem is a direct consequence of combining Lemma \ref{null:cM} with  time-averaging of \eqref{RSW} and  uniform estimates of Theorem \ref{thm:main}.

\begin{theorem}\label{thm:TA}Under the same assumptions as Theorem \ref{thm:main} with $C_\fco'\equiv0$,  
 for any fixed $T\in(0,T_m]$, there exists a function $\Phi:\mM\mapsto\mR$ which is  independent of $p_1$, such that  
\be\label{estimate:RSWthm}\nonumber\begin{split}&\Big\|{1\over T}\int_0^T \vu \,dt-J\grad  \Phi(p_2)\Big\|_{{k-3}}
+\Big\|{1\over T}\int_0^T {\epR\over\epF} h \,dt-\Psi(p_2)\Big\|_{{k-2} }\\\le& C \ep({2M\over T}+M^2).\end{split}\ee
where   constant $M:=\displaystyle\max_{t\in[0,T]}\|(\vu,h)\|_{k}$ and $\Psi(p_2)$ is uniquely determined by $\Psi'(p_2)=\fco(p_2)\Phi'(p_2)$ and $\int_{0}^\pi g_1g_2 \Psi \,dp_2=0$. 
 \end{theorem}
 
 Thus, not only $ {1\over T}\int_0^T \vu \,dt$ can be approximated by a {\bf longitude-independent zonal flow}, but also $ {1\over T}\int_0^Th \,dt$ can be approximated by a longitude-independent height field.
 
  This theoretical result is consistent with many numerical studies and observations.  
For a partial list of computational results, we mention \cite{galperinnakanoetal2004} for 3D models, \cite{vallismaltrud1993, nozawayoden1997} for 2D models, and references therein. Note that many of these computations attempt to simulate turbulent flows with sufficiently high resolutions. Zonal structures in these numerical results are either directly noticeable by naked eyes or after some time-averaging procedures.
On the other hand, we have observed  zonal flow patterns (e.g. bands and jets) on giant planets for hundreds of years, which has attracted considerable interests recently thanks to spacecraft missions and the launch of the Hubble Space Telescope (e.g. \cite{garciasanchez2001}).  There are also observational data in the oceans on Earth showing persistent zonal flow patterns (e.g. \cite{maximenkobangetal2005}).
 
\subsection{Loss of uniform estimates on {\it higher} norms}\label{subs:loss}
The large operator \eqref{def:cM} is skew-self-adjoint in  $L^2$ inner product, but not in terms of higher derivatives, due to the variable coefficient $\fco$ in $\cL_0$. This can potentially prevent us from proving an uniform lower bound on the life span of classical solutions. We give two examples of first order hyperbolic PDEs with variable coefficients in the large operators: the essential difference being Example 1 has first order derivative in the operator  and Example 2 had zero-th order in the operator which is the same as the rotation operator of \eqref{RSW}.  
\medskip

\noindent {\bf Example 1}. This is a two-dimensional shear flow in a two-dimensional torus, spatial domain. The unknown $v(t,x,y)$ satisfies, 
\[
v_t+vv_y+{\sin(y)\over\ep}v_x=0
\]
with initial datum satisfying 
\[
v(0,0,0)=0,\quad v_x(0,0,0)=1\qand v_y(0,0,0)=-1.
\] Clearly $\|v\|_{L^2}$ is conserved in time for classical solutions.  As long as the solution stays smooth, we have $v=0$ remains constant along the characteristic curve $(t,0,0)$.  Along the same curve, the growth rates of $v_x$ and of $v_y$ are given on the right sides below,
\[\ba
(v_x)_t+v(v_x)_y+{\sin(y)\over\ep}(v_x)_x&=-v_xv_y,\\
(v_y)_t+v(v_y)_y+{\sin(y)\over\ep}(v_y)_x&=-(v_y)^2-{\cos(y)\over\ep}v_x.
\ea\]
Thus, as $t$ increases from $0$,  $v_x$ is increasing and positive whereas $v_y$ is decreasing and negative. Therefore, the growth rate of $v_y$ is bounded from above by $-(v_y)^2-{1\over\ep}\cos 0$. Since $v_y$ starts with initial value $-1$, it will  approach $-\infty$ at a positive time no later than $O(\sqrt\ep)$. The validity time interval of classical solutions therefore shrinks to $0$ as $\ep\to0$.\hfill $\square$ 
\medskip

\noindent {\bf Example 2}. This mimics {\it variable} Coriolis force {\it without} pressure gradient. The unknowns $v(t,x,y)$, $u(t,x,y)$ satisfy\[\left\{\begin{aligned}
v_t+vv_y+{\sin(y)\over\ep}u&=0,&&\qquad v(0,0,0)=0,\;\; v_y(0,0,0)=-1,\\
u_t\phantom{+vv_y\,\,}-{\sin(y)\over\ep}v&=0,&&\qquad u(0,0,0)=1, 
\end{aligned}\right.
\]
and one can assume both are independent of $x$, although it is not essential. 
Clearly $\|v\|_{L^2}+\|u\|_{L^2}$ is conserved in time for classical solutions.  As long as the solution stays smooth, we have $v=0$ remains constant along the characteristic curve $(t,0,0)$.  Along the same curve,  the growth rates of $v_y$ are given on the right side below,
\[
(v_y)_t+v(v_y)_y=-(v_y)^2-{\cos(y)\over\ep}u--{\sin(y)\over\ep}u_y.
\]
But $u$ remains constant $1$ and $y$ remains $0$. Therefore,
$v_y(0,0,0)$ (i.e. divergence of $(u,v)$) tends to $-\infty$ in $O(\sqrt\ep)$ time.
 \hfill $\square$ 

\bigskip

The rest of this article is organised as follows. In section \ref{sec2} we find a corrector $\cN$ to the Laplacian so that $\Delta-\cN$ commutes with the large operator $\cL$. In section \ref{sec3} we characterize the kernel of $\cL$, some projection onto this kernel and show that the time averages of solutions stay close to zonal flows. In section \ref{sec:geo} we discuss geometry of the surface $\mM$ in elementary terms.

\bigskip
  
  \section{\bf Commutator}\label{sec2}
The dot product of vectors will be denoted by a dot, e.g. $\vu\cdot\vv$ (c.f. \S \ref{sec:geo} for detailed information). The $L^2(\mM)$ inner product of vector fields will be denoted by $\langle\,,\,\rangle$ namely
\[
\langle \vu,\vv\rangle=\int_\mM\vu\cdot\vv.
\]
Let sup-script ${}^*$ denote the $L^2(\mM)$-adjoint of the operator it attaches to. In fact, all occurrences of (skew)-adjointness are with respect to the $L^2(\mM)$ inner product, unless noted otherwise.

We note that the product rule for $\grad(\vu\cdot\vv)$, already complicated in flat geometry, is even more so on a surface. Thus, we avoid using it altogether here.

\subsection{Hodge decomposition}\label{sec:Hodge}

We aim to use differential geometric tools in an elementary fashion. More details are given in elementary terms in  \S \ref{sec:geo} and in particular, we know singularity caused by longitude-colatitude parametrisation is removable.

For a scalar field $h$ defined on $\mM$, we define the gradient $\nabla h$ as the result of projecting the $\mR^3$ gradient of $h$ onto the tangent bundle of $\mM$.

Since apparently $J^2=-1$ and $\langle J\vu_1,J\vu_2\rangle=\langle \vu_1,\vu_2\rangle$, we have  
\be\label{J:skew}
J^*=-J=J^{-1}.
\ee
Then, define $\dv$ to be the skew-adjoint of $\grad$  and $\curl$ to be the skew-adjoint of $J\grad$,
\be\label{eq:gradperp0}
\dv:=-\grad^*,\qquad\curl:=-(J\grad)^*=-\dv J,
\ee
both of which map vector fields to scalar fields.

The following properties then hold regardless of the geometry of $\mM$,
\begin{align}
\label{eq:gradperp1}&\dv\grad=\curl(J\grad)=\Delta,\\
\label{eq:gradperp2}&\dv(J\grad)=\curl\grad=0.
\end{align}

Scalar Laplacian and vector Laplacian are both denoted by the same symbol $\Delta$ which shall be understood as the Laplacian acting on whatever field that follows. Its action on scalar fields equals the first two expressions of \eqref{eq:gradperp1} and its action on vector fields is the well-known surface Laplacian
\be\label{def:surf:Lap}
\Delta\vu=\grad\dv\vu+J\grad\curl\vu
\ee
 Then, it is straightfoward to show $\Delta$ commutes with $\grad,\rgrad,\dv,\curl$.  Immediately,
\be\label{comm:L:pa}
[\Delta,\cL_\pa]\equiv0\quad\text{ and therefore }\quad [\Delta,\cL]\Vcol={1\over\epR}\bpm[\Delta,\fco J]\vu\\0\epm.
\ee

By the theory of Hodge decomposition,  any smooth tangent vector field on a manifold in the same cohomology class as the 2-sphere is uniquely the sum of an irrotational and an incompressible vector fields. Using (pseudo-) differential operators, this is
\be\label{Hodge}
\ba
\vu&:=\grad\Deltainv \dv\vu+J\grad\Deltainv \curl\vu\\
&=\grad\Deltainv \dv\vu+J^{-1}\grad\Deltainv \dv J\vu
\ea
\ee
The inverse Laplacian is unique up to an additional constant whose exact value is immaterial in the above expressions and throughout this article. Because of this, from now on, we assume
\[
\text{The result of }\Delta^{-1} \text{ acting on a scalar field has zero global mean.}
\]

One can show  (\cite{Taylor:I}) there exist constants $c,C$ that only depend on the domain and the value of integer $k$ so that $c\|\vu\|_k\le \|\grad\Deltainv \dv\vu\|_k+\|\rgrad\Deltainv \curl\vu\|_k\le C\|\vu\|_k$. By integrating by parts and the fact that $\dv\vu$ has zero global mean, we have $\|\grad\Deltainv \dv\vu\|_0^2\le \|\Deltainv \dv\vu\|_0\cdot\|\dv\vu\|_0\le \|\dv\vu\|_0\cdot\|\dv\vu\|_0$ and similarly on $\|\rgrad\Deltainv \curl\vu\|_0$, we then redefine vector-field norms using scalar-field norms,
\be\label{def:Hk}
\|\vu\|_k:=\left(\|\dv\vu\|_{k-1}^2+\|\curl\vu\|_{k-1}^2\right)^{1\over2}\qfor k\ge1.
\ee
 This then induces the definition of $H^k$ inner product for vector field
\[
\langle \vu,\vu' \rangle_k:= \langle \dv\vu,\dv\vu' \rangle_{k-1}+\langle \curl\vu,\curl\vu' \rangle_{k-1}\qfor k\ge1.
\]

More definitions and relevant properties can be found in \cite[Appendix A]{ChMa:Euler:sphere}.  


\subsection{Finding corrector to commutator}
In order to obtain uniform-in-$\ep$ estimate of the $H^k$ norm of the solution for a time period that is bounded from below uniformly in $\ep$,  we endeavor to find differential operators that commute with the large operator $\cL$.  Due to the useful knowledge on the scalar/vector Laplacian operators given in \S \ref{sec:Hodge},  we aim to find pseudo-differential operator $\cN$ of order less that two so that $[\cN,\cL]=[\Delta,\cL]$ (or in approximate sense). Due to the fact that $\cL$ is skew-self-adjoint and Laplacian is self-adjoint, this is equivalent to $[\cN^*,\cL]=[\Delta,\cL]$ and adding it back shows that it is equivalent to finding a self-adjoint operator $\cN$.  

 In view of \eqref{comm:L:pa}, this motivates us to analyze commutator $[\Delta,\fco J]\vu$.
First, we define
 \[
\cA\vv:=(J\grad\fco)\dv\vv\quad\text{ so that }\quad \cA^*\vv=\grad\big\{(\grad\fco)\cdot (J\vv) \big\},
\]
since $\grad^*=-\dv$ from \eqref{eq:gradperp0}. Define the symmetric part of $\cA$ times 2,
\[
\cAs:=\cA+\cA^*.
\]
For any linear operator mapping 2-vector fields onto itself, for example $\cA$, we let $\fCJ(\cA)$ denote its conjugation via rotation $J$,
\[
\fCJ(\cA):=J^{-1}\cA J
\]
Since $J^{-1}=-J=J^*$, we have
\begin{subequations}\label{prop:conj}
\begin{align}
\label{conj1}\fCJ^2(\cA)&=\cA,\\
\label{conj3}\big(\fCJ(\cA) \big)^*&=\fCJ(\cA^*),
\end{align}
\end{subequations}

\begin{lemma}\label{lem:De:fJ}
\[
[\Delta,\fco J]\vu=(1+\fCJ)\big(\cAs  \big)\vu.
\]
\end{lemma}

\begin{proof}
By the second line of Hodge decomposition \eqref{Hodge}, the definition of $\fCJ$ and the apparent fact that scalar multiplication commutes with $J$,
\be\label{pr:comm:1}
\ba{}
[\Delta,\fco J]\vu&=[\grad\dv,\fco J]\vu+ \big[\fCJ(\grad\dv),\fco J\big]\vu\\
&=[\grad\dv,\fco J]\vu+\fCJ\big( [\grad\dv,\fco J]\big)\vu.
\ea
\ee
 
Using Calculus rules on Riemannian manifold we show
\[
\ba{}
[\grad\dv,\fco J]\vu&=\grad\dv(\fco J\vu)-\fco J\grad \dv \vu\\
&=\grad\Big\{(\grad\fco)\cdot(J\vu) +  \fco\dv(J\vu) \Big\}-J\Big\{\grad\big(\fco\dv\vu\big)-(\grad\fco)\,\dv\vu\Big\}\\
& =\grad\Big\{(\grad\fco)\cdot(J\vu)\Big\}+(J\grad\fco)\,\dv\vu-J\grad\big(\fco\dv\vu\big) + \grad\big(\fco\dv(J\vu) \big) \\
& =\cA^*\vu +\cA\vu-J\grad\big(\fco\dv\vu\big) + \fCJ\left(J\grad\big(\fco\dv\big) \right)\vu.
\ea
\] 
Combining this with \eqref{conj1}, we carry on from \eqref{pr:comm:1} and complete the proof.
\end{proof}

The possible candidate for the corrector $\cN$ is chosen as follows. For scalar functions $a, b$,  define zero-th order, self-adjoint operators,
\be\label{N:both}
\ba
\cN_{di}\Vcol:=\bpm a\,\vu\\a\, h\epm,&\qquad \cN_{ad}\Vcol:=\bpm \big(J\grad b\big)h\\\big(J\grad b\big)\ccdot\vu \epm,
\ea
\ee
where the subscrpit ``$di$'' indicates diagonal and ``$ad$'' anti-diagonal. They are not the most general choices, but will suffice our purpose of finding at least one commutator. 
Apparently
\be\label{comm:easy}
 [\cN_{di},\cL_0]\equiv0,\quad 
[\cN_{di},\cL_\pa]\Vcol=-\bpm h\grad a\\\vu\cn a\epm\qand 
[\cN_{ad},\cL_0]\Vcol=\bpm \fco h\grad b\\\fco\vu\cn b\epm.
\ee

Recall we look to satisfy requirement $[\cN,\cL]=[\Delta,\cL]$ (exactly or approximately) where the nontrivial term of the right hand side is essentially $[\Delta,\fco J]\vu$, which according to Lemma \ref{lem:De:fJ},  includes first order spatial derivatives of $\vu$.  Since all three commutators in \eqref{comm:easy} contains only zero-th order derivatives, the first derivatives in $[\Delta,\fco J]\vu$ can only be balanced by
\be\label{com:N:L}\ba{}
[\cN_{ad},\cL_\pa]\Vcol&=
\bpm (J\grad b)(\dv\vu)\\(J\grad b)\cdot\grad h\epm-\bpm \grad\big\{(J\grad b)\cdot\vu\big\}\\\dv\big((J\grad b) h\big)\epm\\
&=\bpm (J\grad b)(\dv\vu) +\grad\big\{(\grad b)\cdot(J\vu)\big\}\\0\epm
\stackrel{def}{=}\bpm{\cA^s_b(\vu)}\\0\epm,
\ea\ee
where we used $(J\vv_1)\cdot\vv_2=-\vv_1\cdot(J\vv_2)$. Also the last $0$ results from combining the product rule and $\dv(J\grad b)\equiv0$. Thus,
the velocity component $\cA^s_b(\vu)$ is defined the same way as $\cAs(\vu)$ only with $\fco$ replaced by $b$. But simply choosing $b$ as some constant times $\fco$ can not exactly balance the $[\Delta,\fco J]\vu$ term as Lemma \ref{lem:De:fJ} reveals. Further computation is needed.


Let $\vu=\grad\sigma_1+J\grad\sigma_2$ and $\ti\vu=\grad\ti\sigma_1+J\grad\ti\sigma_2$ with the later being the ``test function''. Then, 
\[
\Big\langle\ti\vu,\cAs(\vu)\Big\rangle=\Big\langle\grad \ti\sigma_1,\cAs(\grad \sigma_1)\Big\rangle+\Big\langle \grad \ti\sigma_2, \fCJ(\cAs)(\grad \sigma_2)\Big\rangle
+\Big\langle \grad \ti\sigma_1,\cAs(J\grad \sigma_2)\Big\rangle+\Big\langle \grad \ti\sigma_2,J^{-1}\cAs(\grad \sigma_1)\Big\rangle.
\]
By similar calculation, replacing $\cAs$ with $\fCJ(\cAs)$ and noting \eqref{conj1}, we have
\[
\Big\langle\ti\vu,\fCJ(\cAs)(\vu)\Big\rangle=\Big\langle \grad \ti\sigma_1,\fCJ(\cAs)(\grad \sigma_1)\Big\rangle+\Big\langle\grad \ti\sigma_2,\cAs(\grad \sigma_2)\Big\rangle
+\Big\langle \grad \ti\sigma_1,J\cAs(\grad \sigma_2)\Big\rangle+\Big\langle \grad \ti\sigma_2,\cAs(J^{-1}\grad \sigma_1)\Big\rangle.
\]
By  $\Big\langle \grad \ti\sigma,\fCJ(\cAs)(\grad \sigma)\Big\rangle=\Big\langle J\grad \ti\sigma,\cA(J\grad \sigma)+\cA^*(J\grad \sigma)\Big\rangle$ and the definitional fact $\cA(J\grad )=0$, we have $\Big\langle \grad \ti\sigma_i,\fCJ(\cAs)(\grad \sigma_i)\Big\rangle=0$ ($i=1,2$). Also, $\cA(J\grad)=0$ allows us to cancel parts of the  cross terms namely those products involving a ``1'' sub-script and a ``2'' sub-script. Therefore
\begin{align*}
\Big\langle\ti\vu,\cAs(\vu)\Big\rangle&=\Big\langle\grad\ti \sigma_1,\cAs(\grad \sigma_1)\Big\rangle
+\Big\langle \grad \ti\sigma_1,\cA^*(J\grad \sigma_2)\Big\rangle+\Big\langle \grad \ti\sigma_2,J^{-1}\cA(\grad \sigma_1)\Big\rangle,\\
\Big\langle\ti\vu,\fCJ(\cAs)(\vu)\Big\rangle&=\Big\langle\grad\ti \sigma_2,\cAs(\grad \sigma_2)\Big\rangle
+\Big\langle \grad \ti\sigma_1,J\cA(\grad \sigma_2)\Big\rangle+\Big\langle \grad \ti\sigma_2,\cA^*(J^{-1}\grad \sigma_1)\Big\rangle
\end{align*}
Subtract the first equation from the second one and substitute Lemma \ref{lem:De:fJ} into the left hand side; on the right side, move the $\nabla$ acting on the first factor of each inner product to the $\nabla^*$ acting on the second factor, noting \eqref{J:skew}, to obtain, for $\vu=\grad\sigma_1+J\grad\sigma_2$ and $\ti\vu=\grad\ti\sigma_1+J\grad\ti\sigma_2$ (with all $\sigma$'s set to have zero mean over $\mM$)
\be\label{GH0}
\Big\langle\ti\vu,[\Delta,\fco J]\vu\Big\rangle-2\Big\langle\ti\vu,\cAs(\vu)\Big\rangle
=-\langle\ti\sigma_1,\cG(\sigma_1) \rangle+\langle\ti\sigma_2,\cG(\sigma_2) \rangle+\langle\ti\sigma_1,\cH(\sigma_2) \rangle+\langle\ti\sigma_2,\cH(\sigma_1) \rangle,
\ee
where 
\[
\cG:=\grad^*(\cA+\cA^*)\grad,\qquad
\cH:=\grad^*\big(J\cA+(J\cA)^*\big)\grad,
\]
with $\grad^*=-\dv$ from \eqref{eq:gradperp0}. Setting the testing function $\ti\vu=\grad\ti\sigma_1$ and $\ti\vu=J\grad\ti\sigma_2$ respectively in \eqref{GH0} yields,
\be\label{GH}
\ba
-\dv\big([\Delta,\fco J]\vu-2 \cAs(\vu)\big)&=-\cG(\sigma_1)+\cH(\sigma_2),\\
-\curl\big([\Delta,\fco J]\vu -2\cAs(\vu)\big)&=\cG(\sigma_2)+\cH(\sigma_1).
\ea
\ee

\begin{lemma}\label{lem:exact:com}
On the surface of revolution $\mM$ parametrized by \eqref{surface:par} and equipped with metric tensor \eqref{metric:gij}, let the Coriolis parameter $\fco$ to be independent of longitude namely $\pa_{1}\fco=0$. Choose
\[
a=\left({\epF\over\epR}\right)^2\fco^2\qand b={2\epF\over\epR}\fco \quad \text{in the definitions of }\; \cN_{di},\, \cN_{ad}\, \text{ in }\, \eqref{N:both}
\]
and let the corrector operator
\[
\cN=\cN_{di}+\cN_{ad}.
\]
\begin{enumerate}[(i)]
\item \label{L1} If $\pa_{2}\fco=\sqrt{g_1g_2}$ then $\cG\equiv\cH\equiv0$ and   the commutation $[\cN,\cL]=[\Delta,\cL]$ holds.
\item \label{L2} Let integer $k\ge0$. If $\fco\in C^k(\mM)$ and
\be\label{cond:Cfco}
\Big\|{(\pa_2\fco)\sqrt{g_1g_2}-C_\fco g_1g_2\over g_2}\Big\|_{k}\le C_\fco'\ep
\ee
for some constants $C_\fco, C_\fco'$, then  the  corrector $\cN$ defined above 
satisfies
\[
\left\|\big[\Delta-\cN, \cL \big]\Vcol\right\|_k\le C_kC_\fco'\,\|\vu\|_k.
\]
\end{enumerate}
\end{lemma}
Note for a perfect sphere i.e. $R(p_2)\equiv1$ and for the usual choice of Coriolis parameter in geophysics namely $\fco=\cos(p_2)$, the condition \eqref{cond:Cfco} is met with $C_\fco=-1$ and $C_\fco'=0$.
\begin{proof} (\ref{L1})
For $\cG$, we compute its first part by applying the product rule and $\dv(J\grad\fco)=0$, \[
\grad^*\cA\grad\sigma=-\dv\big((J\grad\fco)\Delta\sigma\big)=-(J\grad\fco)\cn\Delta\sigma.
\]
Similarly, it is straightforward to show that   operator $(J\grad\fco)\cdot\nabla$ is skew-self-adjoint. Thus,
\be\label{G:comm}
\cG=\left[\Delta,(J\grad\fco)\cdot\nabla\right].
\ee
Further,
by the gradient formula in \eqref{grad:pj}, the fact that $\fco$ is independent of $p_1$,  the fact that $J$ is clockwise rotation so that $ J\vv_{\pa_2}=-{|\vv_{\pa_2}|\over |\vv_{\pa_1}|}\vv_{\pa_1}=-\sqrt{g_2\over g_1}\vv_{\pa_1}$ and the directional derivative formula \eqref{dot:grad}, we have
$G=\left[\Delta\,,\,-{\pa_2\fco\over\sqrt{g_1g_2}}\,\pa_1\right]$.
Since  the coefficients in the Laplacian \eqref{Lap:local} are independent of $p_1$,  in view of \eqref{G:comm} and the assumption of part (\ref{L1}), we have proven $\cG\equiv0$.

Next, we compute the entirety of $\cH$,
\be\label{H:prod}
\ba
\cH(\sigma)&=\dv\big((\grad\fco)\Delta\sigma\big)-\Delta\big((\grad\fco)\cn\sigma)\\
&=\dv\big(\grad\fco\big)(\Delta\sigma)+\big[(\grad\fco)\cn\,,\,\Delta\big]\sigma\qquad\text{(by product rule)}
\ea
\ee
Again, the assumption of part (\ref{L1}) yields
$\grad \fco=\sqrt{g_1\over g_2}\,\vv_{\pa_2}$ and so by the divergence formula \eqref{dv:local},
\[\ba
\cH(\sigma)&=\dv\big(\sqrt{\tfrac{g_1}{g_2}}\vv_{\pa_2}\big)\Delta \sigma+\big[\sqrt{\tfrac{g_1}{g_2}}\pa_2\,,\,\Delta\big]\sigma \\
&=\big(\tfrac{1}{\sqrt{g_1g_2}}\pa_2g_1\big)\Delta\sigma+\big[\sqrt{\tfrac{g_1}{g_2}}\pa_2\,,\,\tfrac{1}{g_1}\big](g_1\Delta\sigma)+\tfrac{1}{g_1}\big[\sqrt{\tfrac{g_1}{g_2}}\pa_2\,,\,{g_1}\Delta\big]\sigma  
\ea
\]
where the last two terms result from simple manipulation of commutation.
By the local expression of Laplacian \eqref{Lap:local} and the fact that $g_1,g_2$ are independent of $p_1$, the last commutator vanishes. Since apparently $\big[\sqrt{\tfrac{g_1}{g_2}}\pa_2\,,\,\tfrac{1}{g_1}\big](g_1\Delta\sigma)=\left\{\sqrt{\tfrac{g_1}{g_2}}\pa_2\big(\tfrac{1}{g_1}\big)\right\}(g_1\Delta\sigma)$ and therefore the other two terms cancel exactly, we have shown $\cH\equiv0$. Combining $\cG\equiv\cH\equiv0$ with \eqref{GH} and noting \eqref{Hodge} yields $[\Delta,\fco J]\vu=2 \cAs(\vu)$. Therefore, in view of the commutations \eqref{comm:L:pa}, \eqref{comm:easy}, \eqref{com:N:L}, we complete the proof of part (\ref{L1}).\medskip

(\ref{L2}) The calculation in \eqref{G:comm} and \eqref{H:prod} is independent of the form of $\pa_2\fco$. The exact cancellation only comes in after we apply the assumption of part (\ref{L1}) which essentially is to set $\grad\fco$ to be $\sqrt{g_1\over g_2}\,\vv_{\pa_2}$. Therefore, the linear dependence of $\cG,\cH$ on $\fco$ means that, for part (\ref{L2}), we still can use \eqref{G:comm} and \eqref{H:prod} and then   replace each occurrence of $\grad\fco$ in there by 
\[
\vec\xi:=\grad\fco-C_\fco\sqrt{g_1\over g_2}\,\vv_{\pa_2}=\left({\pa_2\fco-C_\fco\sqrt{g_1g_2}\over g_2}\right)\vv_{\pa_2}
\] 
so that simple functional analysis and derivative counting yields estimates
\[
\|\cG(\sigma)\|_{k-1}+\|\cH(\sigma)\|_{k-1}\le C_k\|\vec\xi\|_{k+1}\|\grad\sigma\|_k.
\]
The $H^{k+1}$ norm of $\vec\xi$, in view of definition \eqref{def:Hk} and the divergent formula \eqref{dv:local}, equals the left hand side of \eqref{cond:Cfco}. Thus,
\[
\|\cG(\sigma)\|_{k-1}+\|\cH(\sigma)\|_{k-1}\le \ep\, C_k C_\fco'\,\|\grad\sigma\|_k. 
\]
Combining this with \eqref{GH} and noting \eqref{def:Hk} yields $\big\|[\Delta,\fco J]\vu-2 \cAs(\vu)\big\|_k\le\ep\, C_k C_\fco'\,\|\vu\|_k$. Therefore, in view of the commutations \eqref{comm:L:pa}, \eqref{comm:easy}, \eqref{com:N:L}, we complete the proof of part (\ref{L2}). 
\end{proof}
\bigskip

\section{\bf Time-average and zonal flows}\label{sec3}

To prove Theorem \ref{thm:TA} in the framework of  \cite{ChMa:TA}, the main task is to identify  the kernel $\nllL$, the operator $\pcL$ which denotes {\it some} projector onto $\nllL$, and to establish  an upper bound on $(h,\vu)-\pcL(h,\vu)$ in terms of $\cL(h,\vu)$. The projection $\pcL$ we will introduce in the next lemma is not necessarily an orthogonal projection and is based on  taking zonal average of the zonal wind component of $\vu$. 

\begin{lemma}\label{null:cM}
On manifold $\mM$, consider sufficiently regular scalar function $h$ with zero global mean and velocity field $\vu$.  
\begin{enumerate}[(i)]
\item\label{N1} $(\vu,h)\in\nll\{\cL\}$ if and only if
\be\label{iff:1}
\dv\vu=  \dv(\fco \vu)=0\qand {\epR\over\epF}\grad h+\grad\Deltainv \dv(\fco J\vu)=0.
\ee
if and only if
\be\label{iff:2}
\left\{\ba
&\mbox{there exists a sufficiently regular function }\Phi:\mM\mapsto\mR\mbox{ which is  independent of $p_1$\, s.t.}\\
&\qquad\quad\vu=J\grad \Phi(p_2) \qquad\mbox{ and }\quad h={\epF\over\epR}\Psi(p_2)\\
&\text{where $\Psi(p_2)$ is uniquely determined by $\Psi'(p_2)=\fco(p_2)\Phi'(p_2)$ and $\int_{0}^\pi g_1g_2 \Psi \,dp_2=0$.}
\ea\right.\ee
\item\label{N2} The following defines a projection operator onto $\nll\{\cL\}$,
\be\label{pcM}\begin{split}\pcL(\vu,h)=&(\vutil\,,-{\epF\over\epR}\Deltainv \dv(\fco J\vutil ))\quad
\mbox{where}\quad\vutil:=\dfrac{\left(\oint_{C(\theta)}\vuinc\cdot \vv_{\pa_1}\right)}{\oint_{C(\theta)}\vv_{\pa_1}\cdot\vv_{\pa_1}}\,\vv_{\pa_1}.\end{split}\ee
Here, $\vuinc=J^{-1}\grad\Deltainv\dv(J \vu)$ is the div-free component in  the Hodge decomposition of $\vu$; and $\oint_{C(\theta)}$ is the line integral along the circle $C(\theta)$ at a fixed colatitude $\theta$.
\item\label{N3} If further $\pa_{2}\fco=\alpha\sqrt{g_1g_2}$ for some constant $\alpha$, then the velocity and height components of the projection defined above satisfy
\be\label{est:pcL}
\|\vu-\vutil \|_{{k}}+\|\frac{\epR}{\epF} h+ \Delta^{-1}\dv( \fco J\vutil)\|_{k+1}\le C \big\| \cL[(\vu,h)]\big\|_{{k+2}}.
\ee
\end{enumerate} 
\end{lemma}

\begin{remark}We can also use $\vu$ instead  of $\vuinc$ in the $\oint_{C(\theta)}$ integral of \eqref{pcM}, knowing that Stokes Theorem guarantees the circulation of $(\vu-\vuinc)$ over any closed path vanishes.\end{remark}

\begin{remark}Projection $\pcL$ it is not an $L^2$-orthogonal projection anymore. Although the $L^2$-orthogonal projection onto $\nllL$ always exists by standard theory of Hilbert space, it is unclear that such a projection satisfies the estimate \eqref{est:pcL} in any $k$ spaces.
\end{remark}

\begin{proof} (\ref{N1})  By definition of $\cL$ in \eqref{def:cM}, 
\be\label{cond:1} (\vu,h)\in\nll\{\cL\}\iff\dv\vu=0\quad\qand {\epR\over\epF}\grad h+\fco J\vu=0.\ee
Apply Hodge decomposition  \eqref{Hodge} to the $\fco J\vu$ term in above,
\[
{\epR\over\epF}\grad h+\grad\Deltainv \dv(\fco J\vu)+J^{-1}\grad\Deltainv \dv(\fco J^2\vu)=0.
\]
Due to the uniqueness of Hodge decomposition, both the irrotational and incompressible parts of the left hand side should vanish, which yields two conditions. Substituting them back to
  \eqref{cond:1} proves the equivalent conditions of $\nllL$ in \eqref{iff:1}.

Since $\dv\vu=0$ if and only if $\vu=J\grad \Phi$ for some scalar function $\Phi$, we apply the product rule on $\dv(\fco \vu)=0$ to find $(\grad\fco) \cdot(J\grad\Phi)=0$ namely the two gradients $\grad\fco$ and $\grad\Phi$ are parallel at every point of $\mM$. Since $\fco$ is independent of $p_1$ which makes $\grad\fco$ have only $\vv_{\pa_2}$ component, $\Phi$  should also be independent of $p_1$.
 
Therefore, in the last equality of \eqref{iff:1},  the term $\fco J\vu$  equals $-\fco\grad\Phi$ which leads to
 \[\ba
 {\epR\over\epF}\grad h&=\grad\Deltainv \dv(\fco(p_2)\grad\Phi(p_2))\\
 &=\grad\Deltainv \dv\left(\fco(p_2)\pa_2\Phi(p_2) {\vv_{\pa_2}\over g_2}\right)&&\quad\text{(by \eqref{grad:pj})}\\
 &=\grad\Deltainv \dv\left(\pa_2\Psi(p_2) {\vv_{\pa_2}\over g_2}\right)&&\quad\text{(by assumptions)}\\
 &=\grad\Deltainv \dv\left(\nabla\Psi(p_2)  \right)&&\quad\text{(by \eqref{grad:pj} again)}\\
 &=\nabla\Psi(p_2).
 \ea\]
  And, because $h$ is always of zero global mean, so should $\Psi(p_2)$ be. Thus, we have proven  the equivalent conditions of $\nllL$ in \eqref{iff:2}.\medskip
  
  (\ref{N2}) By the divergence formula \eqref{dv:local}, the longitude-independent zonal flow $\vutil$ defined in \eqref{pcM} is divergence-free. By the same reason and the fact that $\fco$ is also independent of $p_1$, we have $\dv(\fco\vutil)=0$ namely $\curl(\fco J\vutil)=0$. And the last condition of \eqref{iff:1} is directly enforced by the definition \eqref{pcM}. Therefore, $\pcL(\vu,h)\in \nllL$.   
  
  Straightforward calculation can show   \(\pcL\circ\pcL=\pcL.\) Therefore, $\pcL$ is a projection onto $\nllL$.\medskip
  
  (\ref{N3}) 
Denote the incompressible and irrotational parts of the Hodge decomposition \eqref{Hodge}
\[
\vuinc=J^{-1}\grad\Deltainv\dv(J \vu)\qand \vuirr=\grad\Deltainv\dv\vu
\]
so that $\vu=\vuinc+\vuirr$.
  Without loss of generailty, assume
\[\big\| \cL[(\vu,h)]\big\|_{{k+2}}=1.\]

  Immediately, by the definition of $\cL$ in \eqref{def:cM}, 
\be\label{est:cM:2}\|\frac{\epR}{\epF}\grad h+\fco J\vu\|_{{k+2}}+\|\dv\vu\|_{{k+2}}\le 1.\ee
In view of the definition of $H^k$ norm in \eqref{def:Hk}, this implies
  \be\label{est:vuirr}\|\vuirr\|_{{k+3}}=\|\dv\vu\|_{{k+2}}\le 1.\ee

Next, by Hodge decomposition  
\[\frac{\epR}{\epF}\grad h+\fco J\vu=\frac{\epR}{\epF}\grad h+\grad\Deltainv \dv(\fco J\vu)+J^{-1}\grad\Deltainv \dv(\fco J^2\vu).\] 
Apply $\dv$ and $\curl$ respectively and estimate the left hand side using \eqref{est:cM:2} to obtain
\be\label{est:cM:rgrad}\| \dv(\fco \vu)\|_{{k+1}}\le 1,\ee
\be\label{est:cM:grad}\|\frac{\epR}{\epF}\Delta h+\dv( \fco J\vu)\|_{{k+1}}\le 1.\ee

Apply Lemma \ref{est:inc} and estimates \eqref{est:vuirr}, \eqref{est:cM:rgrad},
\[\begin{aligned} \|\vuinc-\vutil\|_{k}&\le \| \dv(\fco \vuinc)\|_{{k+1}} \le\| \dv(\fco \vu)\|_{{k+2}}+\| \dv(\fco \vuirr)\|_{{k+2}} \le C.\end{aligned}\]
Combine with estimate \eqref{est:vuirr} again to obtain
\be\label{est:vu:vut}\|\vu-\vutil\|_{k}\le C.\ee

Finally, 
 since  $h$ and $\Delta^{-1}$ have  zero-global-mean, by the Poincar\'{e} inequality and  the triangle inequality 
\[\|\frac{\epR}{\epF} h+ \Delta^{-1}\dv( \fco J\vutil)\|_{k+1}\le \|\frac{\epR}{\epF}\Delta h+ \dv( \fco J\vutil)\|_{k-1}
\le \|\frac{\epR}{\epF}\Delta h+ \dv( \fco J\vu)\|_{k-1}+\|\vu-\vutil\|_{k}.
\]
In view of  estimates\eqref{est:cM:grad}, \eqref{est:vu:vut}, this finishes the proof of part (\ref{N3}).
  
\end{proof}

\begin{lemma}\label{est:inc}
On manifold $\mM$ with Coriolis parameter satisfying $\pa_{2}\fco=\alpha\sqrt{g_1g_2}$ for some constant $\alpha$, any sufficiently regular, \emph{incompressible} velocity field $\vuinc$ satisfies
\[
\|\vuinc-\vutil \|_k\le {C\over|\alpha|}\| \dv(\fco \vuinc)\|_{{k+1}}
\]
with $\vutil$ defined in \eqref{pcM} as the zonal mean of $\vuinc$. 
\end{lemma}
\begin{proof}
Let $\vuinc=J\grad\Phi$ so that by  the gradient formula in \eqref{grad:pj},
\[
\vuinc={1\over\sqrt{g_1g_2}}\left(\pa_1\Phi\,\vv_{\pa_2}-\pa_2\Phi\,\vv_{\pa_1}\right)
\] Then, combine this with \eqref{pcM} to have
\[
\vutil=\dfrac{-{1\over\sqrt{g_1g_2}}\int_0^{2\pi}\pa_2\Phi(p_1,p_2)\,dp_1}{2\pi}\vv_{\pa_1} 
={1\over2\pi}\pa_2\int_0^{2\pi}\Phi(p_1,p_2)\,dp_1\,{J\vv_{\pa_2}\over g_2}.
\]
Therefore,
\[
\vuinc-\vutil=J\grad\left(\Phi-{1\over2\pi}\int_0^{2\pi}\Phi(p_1,p_2)\,dp_1\right)
\]
and thus, by the Poincare inequality
\[
\|\vuinc-\vutil\|_k\le C\|\pa_1\Phi\|_{k+1}
\]
Finally, by the product rule, we have
 \(
 \dv(\fco \vuinc)=-{1\over\sqrt{g_1g_2}}\pa_2 \fco\pa_1\Phi
 \). Therefore, by the given assumption,
 \(
 \dv(\fco \vuinc)=-\alpha\pa_1\Phi
 \) and so the proof is complete.
\end{proof}

\bigskip

\section{\bf The geometry of a revolving surface}\label{sec:geo}

Recall the parametrization of $\mM$ in \eqref{surface:par}.
Since the definition of manifold requires the construction of charts, we also impose that there exist constants $0<a_1<a_2<{\pi\over2}<a_3<a_4<\pi$ so that
\be\label{cond:chart}\left\{\quad\ba
&\inf_{[0,a_2]}{d\over dp_2}\big(R(p_2)\sin p_2\big)>0,\quad\sup_{[a_3,\pi]}{d\over dp_2}\big(R(p_2)\sin p_2\big)<0\\&\sup_{[a_1,a_4]}{d\over dp_2}\big(R(p_2)\cos p_2\big)<0. \ea\right.
\ee
and
\be\label{cond:Rm}
R^{(m)}(0)=R^{(m)}(\pi)=0\qfor m=1,2,\ldots,k,
\ee
where $R^{(m)}(p_2)$ denotes the $m$-th derivative of $R$ with respect to $p_2$.

Conditions \eqref{cond:chart} ensure that $\sqrt{x^2+y^2}=R(p_2)\sin p_2$ is invertible on $p_2\in[0,a_2] \cap [a_3,\pi]$, and $z=\cos p_2R(p_2)$ is invertible on $p_2\in[a_1,a_4]$. Therefore, the entire $\mM$ can be covered by four charts: two overlapping charts whose union cover exactly the strip $p_2\in[a_1,a_4]$ (one uses local coordinates $(p_1,p_2)$ and the other uses $(p_1+2\pi, p_2)$); one chart that covers exactly the north cap $p_2\in[0,a_2]$ and one that covers exactly the south cap $p_2\in[a_3,\pi]$, both of which use local coordinates $(x,y)$. Finally, for $\mM$ to be a {\it differential} manifold, on the north and south caps, the $z$ coordinate must be a smooth function of the designated coordinates $(x,y)$ of the charts therein. Since $z=\pm\sqrt{R^2(p_2)-x^2-y^2}$ is strictly away from 0 in the caps,  this is reduced to the boundedness and continuity of $D R(p_2),D^2(p_2),\ldots D^k R(p_2)$ and their products with  $D^k$ denoting a generic $k$-th order mixed $x,y$ derivatives. By Leibniz ruls,
\be\label{Dk:R}
D^k R(p_2)=\sum_{1\le m\le k\atop {k_1+\cdots+k_m=k\atop k_1\ldots k_m>0}}R^{(m)}(p_2)\prod_{j=1}^mD^{k_j}p_2.
\ee

 Similarly, for $ Q(p_2):=R(p_2)\sin p_2=\sqrt{x^2+y^2}$,
\[ Q'(p_2)D^kp_2+ \sum_{2\le m\le k\atop{k_1+\cdots+k_m=k\atop k_1\ldots k_m>0}}Q^{(m)}(p_2)\prod_{j=1}^mD^{k_j}p_2=D^k Q(p_2)=D^k \sqrt{x^2+y^2}.
\]
Since \eqref{cond:chart} ensures ${1\over  Q'(p_2)}$ is bounded in the north/south caps, by induction on $k$
\[
|D^kp_2|\le c_k\left({\sqrt{x^2+y^2}}\right)^{-(k-1)} 
\]
 over the north/south caps. Use this to bound the right hand side of \eqref{Dk:R}  and obtain
\[
|D^k R(p_2)|\le \sum_{1\le m\le k} c_m|R^{(m)}(p_2)|(\sqrt{x^2+y^2})^{m-k}\quad\text{ in the north/south caps},
\]
which is why we impose condition \eqref{cond:Rm}.

Next, at a given point ${\mathbf p}\in \mM$, the tangent space is defined as the linear space consisting of all ``tangent vectors'' which are fundamentally derivations (mapping $C^\infty(\mM)\mapsto \mR$ that satisfies the product rule) and this definition is {\it intrinsic}, namely independent of any ambient space. In any local coordinates $(p_1,p_2)$ such as the one we just defined, the tangent space is spanned by partial derivatives $
\big\{{\pa\over \pa p_1},{\pa\over \pa p_2}\big\}$. We then {\it intrinsically} define differential forms as an exterior algebra so that a differential $k$-form $\alpha$ at given point ${\mathbf p}$  maps any $k$-tuple of  tangent vectors\footnote{In fact, the $k$-tuple is understood as the wedge product of $k$ tangent vectors. But in this article, we only use the wedge product of differential forms.} at ${\mathbf p}$ to a scalar and the exterior derivative $d$ is the unique linear mapping from any $k$-form to $(k+1)$-form and satisfying the following three axioms  for any 0-form $f$ (i.e. scalar-valued function), any vector (i.e. derivation) field $X=\sum_jX_j{\pa\over\pa p_j}$ and any $k$-form $\alpha$, ,
\be\label{form:axiom}
\begin{aligned}
&df(X)=X(f),\qquad d(df)=0,\qquad d(\alpha\wedge\beta)=d\alpha\wedge\beta+(-1)^k\alpha\wedge d\beta.
\end{aligned}
\ee  

 The central concepts of exterior algebra include the multi-linearity of the forms and the \emph{wedge product} $\wedge$ satisfying, for  $k$-form $\alpha$ and $\ell$-form $\beta$,  
 \[
 \alpha\wedge \beta \text{ is a $(k+\ell)$-form\quad and \quad} \alpha\wedge \beta=(-1)^{k\ell}\beta\wedge\alpha.
 \]
 The complete definitions and list of properties of differential forms can be found in e.g. \cite{Warner:Diff}. 
In particular, the 1-form $dp_i$ in local coordinates satisfies $dp_i(\frac{\pa}{\pa p_j})=\delta_{ij}$ so that ,
 \be\label{dp:local}
 df=\sum_j{\pa f\over\pa p_j}dp_j\qand df\Big(\sum_jX_j\dfrac{\pa}{\pa p_j}\Big)=\sum_jX_j\dfrac{\pa}{\pa p_j}f.
 \ee 
 
For studying Euler equations on a manifold, it is convenient to introduce the vectorial dot product, namely a positive symmetric bilinear form  $T\mM\times T\mM\mapsto \mR$ which is an intrinsic notion.  In fact, 
at given point ${\mathbf p}\in \mM$, each partial derivative ${\pa\over \pa p_j}$ ($j=1,2$) is identified with a tangent vector which we shall call $\vv_{\pa_j}$. Such identification is denoted by 
\[
{\pa\over \pa p_j}\sim \vv_{\pa_j}\qfor j=1,2
\]
and satisfies, for scalar field $f$ defined in $\mM$,
\[
{\pa\over \pa p_j}f=\nabla_{\vv_{\pa_j}} f.
\]
Combining this with the definition
\[\nabla_{\vv_{\pa_j}}f:={\pa\over\pa s }f(\gamma( s ))\Big|_{ s =0}\quad\text{for any curve }\gamma( s )\subset \mM \text{ \; satisfying \; }\gamma(0)={\mathbf p},\;\; \gamma'(0)=\vv_{\pa_j},\]
we obtain
\[{\pa\over\pa  s }f(\gamma( s ))\Big|_{ s =0}={\pa\over\pa p_j}f(p_1,p_2)\quad\text{for any curve $\gamma( s )$ satisfying the above}.\]
Therefore, we choose $\gamma( s )$ to have its $p_j$ coordinate  to be $(p_j+ s )$ while holding the other coordinate fixed. Then, by the parametrization \eqref{surface:par} of the surface, for $j=1$, we express $\gamma( s )$ in Cartesian coordinates of $\mR^3$ as
\[\begin{aligned}
&\bpm
\cos (p_1+ s )\sin p_2\\\sin (p_1+ s )\sin p_2\\\cos p_2
\epm R(p_2)
 \text{ \; making \; } {\pa\over \pa p_1}\sim\vv_{\pa_1}=\bpm
-\sin p_1\sin p_2\\\cos p_1\sin p_2\\0
\epm R(p_2)
\end{aligned}\]
and for $j=2$, we express $\gamma(s)$ in Cartesian coordinates of $\mR^3$ as
\[\begin{aligned}
& \bpm
\cos p_1\sin (p_2+ s )\\\sin p_1\sin (p_2+ s )\\\cos (p_2+ s )
\epm R(p_2+ s )
 \text{ \; making \; } {\pa\over \pa p_2}\sim\vv_{\pa_2}=
\bpm
\cos p_1\cos p_2\\\sin p_1\cos p_2\\-\sin p_2
\epm R(p_2)+
\bpm
\cos p_1\sin p_2\\\sin p_1\sin p_2\\\cos p_2
\epm R'(p_2).
\end{aligned}\]

For physicality, we require the dot product on the surface to be consistent with that of the ambient Euclidean space $\mR^3$.
That is, the inner product $T\mM\times T\mM\mapsto \mR$ inherits the definition of the ambient $\mR^3$ dot product and can be fully and uniquely represented by the metric tensor $\sg=\{g_{ij}\}$ in the form of a 2-by-2 matrix,
\be\label{metric:gij}
\ba
&\sg:=\left\{\big(\vv_{\pa_i}\cdot\vv_{\pa_j}\big)_{\mR^3}\right\}=
\begin{pmatrix} g_1(p_2)& 0
\\0& g_2(p_2)\end{pmatrix}\\
&\text{where }\;\;
g_1(p_2):=\sin^2(p_2)R^2(p_2)\qand
g_2(p_2):=R^2(p_2)+\left(R'(p_2)\right)^2.
\ea
\ee

Note that the basis $\vv_{\pa_1},\vv_{\pa_2}$ are everywhere orthogonal but not normalised, even if $\mM$ is a perfect sphere. Compared to the choice of an orthonormal basis, this choice will change the expressions for $\nabla$ and $\dv$, but as long as the same $(p_1,p_2)$ coordinates are used, the expressions for the area form in the surface integral \eqref{a:ele} and scalar Laplacian \eqref{Lap:local} remain \emph{un}changed. The particular type of product in \eqref{dot:grad} also remains \emph{un}changed (as an intrinsic property of inner product).
The area form in the $(p_1,p_2)$ coordinates is $\sqrt{|\sg|}dp_1dp_2$ i.e. the surface integral is 
\be\label{a:ele}
\int_0^{2\pi}\int_0^\pi\text{(integrand)}\sqrt{|\sg|}\,dp_1dp_2\quad\text{where}\quad|\sg|=\big|\det\{g_{ij}\}\big|=g_{1}g_{2}.
\ee

For a scalar field $f$,  its gradient can be defined using the so-called musical morphisms $\flat,\sharp$ the details of which can be found in [...]. For the calculation herein, it suffices to acknowledge the property
\be\label{dot:grad} 
\vv\cdot(\grad f)=\pa_\vv(f)=v_1{\pa\over\pa p_1}f+v_2{\pa\over\pa p_2}f\qfor \vv=v_1 \vv_{\pa_1}+v_2  \vv_{\pa_2}.
\ee
Express $\grad f$ in the local basis $\{\vv_{\pa_1},\vv_{\pa_2}\}$ with undetermined coefficients, substitute it into the left hand side and apply the dot product prescribed in \eqref{metric:gij} to find
\be\label{grad:pj}
\grad f=\sum_{j=1}^2\Big({1\over g_{j}}\,{\pa\over\pa p_j}f\Big)\vv_{\pa_j}.
\ee
By duality $\dv=-\grad^*$, \eqref{dot:grad} and integral form \eqref{a:ele}, we must have $\int\big(v_1{\pa\over\pa p_1}f+v_2{\pa\over\pa p_2}f\big)\sqrt{|\sg|}dp_1dp_2=-\int(\dv\vv)f \sqrt{|\sg|}dp_1dp_2$. Therefore,
\be\label{dv:local}
\dv\vv=\sum_{j=1}^2{1\over \sqrt{|\sg|} }\,{\pa\over\pa p_j}\big({\sqrt{|\sg|}}v_j\big)\qfor \vv=v_1 \vv_{\pa_1}+v_2  \vv_{\pa_2}
\ee
and therefore use $\Delta f=\dv(\grad f)$ to obtain
\be\label{Lap:local}
\Delta f=\sum_{j=1}^2{1\over \sqrt{|\sg|} }\,{\pa\over\pa p_j}\Big({\sqrt{|\sg|} \over g_j}\,{\pa\over\pa p_j}f\Big)={1\over g_1}\Big(\pa_1^2+\big(\sqrt{g_1\over g_2}\,\pa_2\big)^2\Big)f
\ee
since $|\sg|=g_1g_2$ and  all terms in the metric tensor \eqref{metric:gij} are independent of $p_1$.

\bigskip
\section*{\bf Acknowledgement}
BC would like to thank Brett Kotschwar and Paul Skerritt for their valuable comments on Calculus on Riemannian manifold.

\end{document}